\documentclass[12pt,thmsa]{article}
\usepackage{amsmath, latexsym, amsfonts, amssymb, amsthm, amscd}

\newtheorem{theorem}{Theorem}

\newtheorem{lemma}{Lemma}
\newtheorem{rem}{Remark}

\newcommand{\e}{\Bbb{E}}

\title{
\textbf{A direct method for solving optimal stopping problems for L\'evy processes}}

\author{\textbf{E.J. Baurdoux\footnote{Department of Statistics, London School of Economics. Houghton street, {\sc London, WC2A 2AE, United Kingdom.} E-mail: e.j.baurdoux@lse.ac.uk}}}
\date{\footnotesize This version: \today}

\begin{document}

\maketitle

\begin{abstract}
\bigskip
\noindent
We propose an alternative approach for solving a number of well-studied optimal stopping problems for L\'evy processes. Instead of the usual method of guess-and-verify based on martingale properties of the value function, we suggest a more direct method by showing that the general theory of optimal stopping for strong Markov processes together with some elementary observations imply that the stopping set must be of a certain form for the optimal stopping problems we consider. The independence of increments and the strong Markov property of L\'evy processes then allow us to use straightforward optimisation over a real-valued parameter to determine this stopping set. We illustrate this approach by applying it to the McKean optimal stopping problem (American put),  the Novikov--Shiryaev optimal stopping problem and the Shepp--Shiryaev optimal stopping problem (Russian option).
\bigskip

\end{abstract}

\vspace{0.5cm}
\section{Introduction}
L\'evy processes have stationary and independent increments and they satisfy the strong Markov property. In this paper we shall not make use of any further properties of L\'evy processes and instead refer the interested reader to the monographs \cite{Be} and \cite{ky}.
 L\'evy processes form a surprisingly rich class with applications in a wide variety of fields including biology, insurance and mathematical finance. In the latter, L\'evy processes have also been popular for studying optimal stopping problems.
 We consider optimal stopping problems of the form
 \[V(x)=\sup_{\tau\in\mathcal{T}}\mathbb{E}_x[e^{-qt} G(X_t)].\]
 Here $G$ is a given function called the pay-off function, $q>0$ is to be thought of as the discount rate and $X$ is a L\'evy process defined on a filtered probability space $(\Omega, \mathcal{F},(\mathcal{F}_t)_{t\geq 0},\mathbb{P}_x)$ satisfying the natural conditions (see \cite{bich}), with $\mathbb{P}_x(X_0=x)=1$ with $\mathbb{E}_x$ denoting the related expectation operator. Finally, $\mathcal{T}$ denotes the set of all $[0,\infty]$-valued stopping times with respect to $(\mathcal{F}_t)_{t\geq 0}$. To exclude trivialities we assume that the paths of $X$ are not monotone.

 For Brownian motion, and, more generally, diffusions, optimal stopping problems are often solved by finding a solution to the related free-boundary problem (see for example p.48-49 in \cite{PS}). For a general L\'evy process, this method might not always be feasible as the infinitesimal generator is now an integro-differential operator, and, unless further assumptions are made on the jump distribution (see for example \cite{Gapeev}) in most cases one resorts to a so-called verification lemma of the following form (with an integrability assumption in case $G$ is unbounded).
Suppose that $\tau^*\in\mathcal{T}$ and denote $V^*(x)=\mathbb{E}_x[e^{-q\tau^*}G(X_{\tau^*})]$. Then the pair $(\tau^*,V^*)$ is a solution to the optimal stopping problem if
$V^*(x)\geq G(x)$ for all $x\in\mathbb{R}$
and if
the process $\{e^{-qt} V^*(X_t)\}_{t\geq 0}$ is a right-continuous supermartingale.
The proof of such a verification lemma is straightforward, see for example Lemma 11.1 in \cite{ky}. However, in general, it is not obvious how $\tau^*$ should be chosen. For continuous processes, a method based on a change of measure was proposed in \cite{BL}, see also \cite{Lerche-Urusov}. This was extended to certain optimal stopping problems for L\'evy processes with one-sided jumps in \cite{Baurdoux}. However, this method hinges on the fact that the underlying process does not overshoot the boundary of the stopping region. We consider general L\'evy processes and the approach we propose boils down to the following elementary steps.
\begin{enumerate}
\item Based on general theory of optimal stopping together with some elementary observations, show that the stopping set can be parameterised by a real number, say $y$, and denote the corresponding stopping time $\tau(y)$.
\item Use the stationarity and independence of the underlying process to find an expression for $V(x,y)$, the expected payoff corresponding to $\tau(y)$ under $\mathbb{P}_x$.
    \item Maximise $V(x,y)$ over $y$ where we are free to choose a convenient value for $x$.
\end{enumerate}
We shall illustrate this method by solving three examples. Firstly, we consider the McKean optimal stopping problem (American put) with $G(x)=(K-e^x)^+$ with $K>0$. Secondly, we study the so-called Novikov--Shiryaev optimal stopping problem with $G(x)=(x^+)^n.$ Finally, we consider the Shepp--Shiryaev optimal stopping problem (Russian option).

\section{The McKean optimal stopping problem}
The value function of the McKean optimal stopping problem (or American put when $q$ is the risk less rate) is given by
\begin{equation}V(x)=\sup_{\tau\in\mathcal{T}}\mathbb{E}_x[e^{-q\tau}\max(K-e^{X_\tau},0)],\label{McKeanosp}\end{equation}
where $K>0$ strike price, $q>0$ discount rate and $X$ the underlying L\'evy process.
This optimal stopping problem was first solved in \cite{McKean} in the case when the underlying is a Brownian motion.
Denote $\underline{X}_t=\inf_{0\leq s\leq t}X_t$  and let $\mathbf{e}_q$ be an independent, exponentially distributed random variable with parameter $q$. Also, for $y\in\mathbb{R}$, denote by $\tau_{y^-}$ the first passage time of $X$ below $y$, i.e. $\tau_y^-=\inf\{t>0 : X_t<y\}$.
In \cite{Mordecki} the random walk proof from \cite{Darling} was extended to the case of a general L\'evy process. In \cite{aky} an alternative proof was given based on a verification lemma. The solution is as follows.
\begin{theorem}
An optimal stopping time for (\ref{McKeanosp}) is given by $\tau^-_{y^*}=\inf\{t>0: X_t <y^*\}$, with $\exp(y^*)= K\mathbb{E}[e^{ \underline{X}_{\mathbf{e}_q}}]$.
\end{theorem}
To prove this we shall make use of the following result, which is (6.33) on p.176 of \cite{ky}. As the proof is straightforward we include it here for completeness.
\begin{lemma}\label{Markov}
For $q,\beta\geq 0$
\[\mathbb{E}_x\left[e^{-q \tau_y^-+\beta X_{\tau_y^-}}1_{\{\tau_y^-<\infty\}}\right]=e^{\beta x}\frac{\mathbb{E}[e^{\beta \underline{X}_{\mathbf{e}_q}}1_{\{-\underline{X}_{\mathbf{e}_q}>x-y\}}]}{\mathbb{E}[e^{\beta \underline{X}_{\mathbf{e}_q}}]}.\]
\end{lemma}
\begin{proof}
Since $1_{\{\underline{X}_{\mathbf{e}_q}<y\}}=1_{\{\tau_y^-<\mathbf{e}_q\}}$ we get by conditioning on $\mathcal{F}_{\tau_y^-}$
\begin{eqnarray*}
\mathbb{E}\left[e^{\beta \underline{X}_{\mathbf{e}_q}}1_{\{\underline{X}_{\mathbf{e}_q}<y\}}\right]
&=&\mathbb{E}\left[e^{\beta \underline{X}_{\mathbf{e}_q}}1_{\{\tau_y^-<\mathbf{e}_q\}}\right]\\
&=&\mathbb{E}\left[1_{\{\tau_y^-<\mathbf{e}_q\}}e^{\beta X_{\tau^-_y}}\mathbb{E}\left[e^{\beta ( \underline{X}_{\mathbf{e}_q}-X_{\tau^-_y})}\left|\mathcal{F}_{\tau_y^-}\right.\right]\right].
\end{eqnarray*}

On the event $\{\tau_y^-<\mathbf{e}_q\}$ and given  $\mathcal{F}_{\tau_y^-}$ it holds that
\begin{eqnarray*}
\underline{X}_{\mathbf{e}_q}-X_{\tau^-_y}&=&\inf_{0\leq s\leq \mathbf{e}_q}(X_s-X_{\tau_y^-})\\
&=&\min\left(X_{\tau_y^-},\inf_{\tau_y^-\leq s\leq \mathbf{e}_q}X_s\right)-X_{\tau_y^-}\\
&=&\inf_{\tau_y^-\leq s\leq \mathbf{e}_q}(X_s-X_{\tau_y^-})\\
&\stackrel{\mathrm{d}}{=}&\inf_{0\leq s\leq \mathbf{e}_q-\tau_y^-}\tilde{X}_s,\\
 \end{eqnarray*}
 and the fact that $\underline{X}_{\tau_y^-}=X_{\tau_y^-}$. Here $\tilde{X}$ denotes an independent copy of $X$.
Furthermore, due to the lack of memory property of the exponential distribution ($\mathbb{P}(\mathbf{e}_q-s>t|\mathbf{e}_q>s)=\mathbb{P}(\mathbf{e}_q >t)$) it follows that
on $\{\tau_y^-<\mathbf{e}_q\}$ we have $\inf_{0\leq s\leq \mathbf{e}_q-\tau_y^-}X_{s}= \underline{X}_{\mathbf{e}_q-\tau_y^-} \stackrel{\mathrm{d}}{=}\underline{X}_{\mathbf{e}_q}.$
Hence
\begin{eqnarray*}\mathbb{E}\left[e^{\beta \underline{X}_{\mathbf{e}_q}}1_{\{\underline{X}_{\mathbf{e}_q}<y\}}\right]
&=&\mathbb{E}\left[1_{\{\tau_y^-<\mathbf{e}_q\}}\e^{\beta X_{\tau^-_y}}\right]\mathbb{E}\left[e^{\beta \underline{X}_{\mathbf{e}_q}}\right]\\
&=&\mathbb{E}\left[e^{-q\tau_y^-+\beta X_{\tau^-_y}}\right]\mathbb{E}\left[e^{\beta \underline{X}_{\mathbf{e}_q}}\right]
\end{eqnarray*}
\end{proof}
\begin{proof}[Proof of Theorem 1]
The proof will follow the three steps as set out in the introduction.
Classical theory of optimal stopping (see for example Corollary 2.9 in \cite{PS}) implies that an optimal stopping time $\tau^*(x)$ exists and is of the form
\[\tau^*(x)=\inf\{t\geq 0: X_t\in D\}\quad \mbox{under }\mathbb{P}_x\]
with $D=\{y\in\mathbb{R}: V(y)=(K-e^y)^+\}$. Since $(K-x)^+$ is convex it follows that $V(\log x)$ is convex well. As $V(\log K)>0$ (consider $\tau=\inf\{t: X_t\leq \log K/2\}$) and $q>0$ this implies that $D=(-\infty, y^*]$ for some $y^*<\log K$.
For $y<\log K$ denote by $V(x,y)$ the expected pay-off corresponding to $\tau^-_y$ when $X_0=x$, i.e.
 \[V(x,y):=\mathbb{E}_x[e^{-q\tau_y^-}\max(K-e^{X_{\tau_y^-}},0)].\]
 From Lemma 1 it follows that
\[V(x,y)=\frac{\mathbb{E}\left[\left(K\mathbb{E}[e^{ \underline{X}_{\mathbf{e}_q}}]-e^{x+\underline{X}_{\mathbf{e}_q}}\right)
1_{\{-\underline{X}_{\mathbf{e}_q}>x-y\}}\right]}{\mathbb{E}[e^{ \underline{X}_{\mathbf{e}_q}}]}
\]
 As we are looking for an optimal choice of $y$ (which we know is independent of $x$) we are looking to maximise
\[\mathbb{E}\left[\left(\mathbb{E}[e^{ \underline{X}_{\mathbf{e}_q}}]-e^{\underline{X}_{\mathbf{e}_q}}\right)
1_{\{Ke^{\underline{X}_{\mathbf{e}_q}}<e^{y}\}}\right].
\]
To maximise this expected value the indicator should be $1$ precisely when the random variable in round brackets is positive. This implies that to $\exp(y^*)= K\mathbb{E}[e^{ \underline{X}_{\mathbf{e}_q}}]$.
\end{proof}
\begin{rem}
For specific classes of L\'evy processes the expression for $x^*$ and the value function become more explicit. For example, see \cite{avram} and \cite{Chan} when the L\'evy process is assumed to have no positive jumps.
\end{rem}

\section{The Novikov--Shiryaev optimal stopping problem}
Next, we consider the pay-off function $G(x)=(x^+)^\nu$ with $\nu>0$, i.e.
\begin{equation}
V(x)=\sup_{\tau\in\mathcal{T}}\mathbb{E}_x[e^{-q\tau} (X_\tau^+)^\nu].\label{NS}
\end{equation}
with $q>0$.
 We shall assume throughout this section that
\begin{equation}
\int_{(1,\infty)} x^\nu\, \Pi(dx)<\infty\label{condnu}
\end{equation}
where $\Pi$ denotes the L\'evy measure of $X$.
This condition is sufficient to guarantee that $\mathbb{E}[\overline{X}^\nu_{\mathbf{e}_q}]<\infty$.

The Novikov--Shiryaev optimal stopping problem was first solved in both a random walk and L\'evy process case in \cite{Novikovshiryaev2} in which the authors extended their results from the setting $\nu\in\mathbb{N}$ in \cite{Novikovshiryaev}. The proof is based is based on a verification lemma.  Note that for a general L\'evy processes (\ref{NS}) was solved in \cite{Surya} for $\nu\in\mathbb{N}$ again using a verification lemma. See also \cite{Salminen}.

The solution to this optimal stopping problems is given in terms of the so-called Appell functions. Here we mention some of their properties without proof and refer to \cite{Novikovshiryaev2} for further details. Appell functions can be defined inductively. For $y>0$ and $s<0$ define
\[Q_s(y)=\frac{1}{\Gamma(-s)}\int_0^\infty u^{-s-1}\frac{e^{-u y}}{\mathbb{E}[e^{-u\overline{X}_{\mathbf{e}_q}}]}\,du\]
and let $Q_0(y)=1$ for any $y>1$.
Then for $s\in(0,\nu)$ we define $Q_s(y)$ via
\[\frac{d}{dx} Q_s(x)=s Q_{n-1}(x)\]
and $\mathbb{E}[Q_s(\overline{X}_{\mathbf{e}_q})]=0$. This expectation can be shown to be finite because of assumption (\ref{condnu}). It follows then that
\begin{equation}\mathbb{E}_x[Q_s(\overline{X}_{\mathbf{e}_q})]=x^s.\label{mean}\end{equation}
We are now ready to state the solution  to (\ref{NS}) as in \cite{Novikovshiryaev2}.
\begin{theorem}
Let $\nu>0$ and suppose $X$ is a L\'evy process satisfying (\ref{condnu}). Then an optimal stopping time for (\ref{NS}) is given by
\[\tau_{a(\nu)}^+=\inf\{t>0: X_t>a(\nu)\}\]
where $a(\nu)$ denotes the positive solution to the equation $Q_\nu(x)=0$.\end{theorem}
\begin{proof}
Again, instead of applying a verification lemma we use a direct approach. We invoke the general theory of optimal stopping to conclude that due to assumption (\ref{condnu}) there exists an optimal stopping time which is given by the first hitting time of the set
\[D=\{x\in\mathbb{R}: V(x)=(x^+)^\nu\}.\]
Note that $D\subset\mathbb{R}^+$ and that $D\neq\emptyset$, again due to (\ref{condnu}).
We follow arguments similar to those in \cite{Christensen} to deduce that for $y>x>0$ (and making explicit the dependence of optimal stopping times on the starting point and using that $\tau^*(y)$ may not be optimal under $\mathbb{P}_x$) it holds that
\begin{eqnarray*}
\frac{V(y)}{(y^+)^\nu}-\frac{V(x)}{(x^+)^\nu}&\leq &\mathbb{E}\left[e^{-q\tau^*(y)}\left(\frac{((y+X_{\tau^*(y)})^+)^\nu}{y^\nu}-\frac{(x+X_{\tau^*(y)})^+)^\nu}{x^\nu}\right)\right]\\
&\leq &\mathbb{E}\left[e^{-q\tau^*(y)}\left(((1+X_{\tau^*(y)}/y)^+)^\nu-((1+X_{\tau^*(y)}/x)^+)^\nu\right)\right]\leq 0.
\end{eqnarray*}
Hence, for any $x\in D$ we have that $V(x)=x^\nu$ and thus also that $y\in D$ when $y>x$. Therefore we conclude that $D=[a^*,\infty)$ for some $a^*>0$.  The strong Markov property, stationarity and independence of increments and (\ref{mean}) now lead to
\begin{equation}\mathbb{E}_x[Q_\nu(\overline{X}_{\mathbf{e}_q)}1_{\{\overline{X}_{\mathbf{e}_q}\geq a\}}]=\mathbb{E}_x[e^{-q\tau_a^+}X_{\tau_a^+}^\nu1_{\{\tau_a^+<\infty\}}]\label{maxthis}\end{equation}
for any $a>0$, where $\overline{X}_t=\sup_{0\leq s\leq t}X_s.$ It suffices now to maximise this over $a$ to find the optimal stopping set $D$.
For this, we refer to Lemma 1 in \cite{Novikovshiryaev2} which states that for each $\nu>0$ there exists $a(\nu)$ such that $Q_\nu(x)\leq 0$ for $0<x<a(\nu), Q_\nu(a(\nu))=0$ and $Q_\nu(x)$ is increasing for $x>a(\nu)$. Just as in the case of the McKean optimal stopping problem, maximising (\ref{maxthis}) over $a$ is now straightforward as we should choose $a$ such that the indicator function is equal to 1 only when $\overline{X}_{\mathbf{e}_q}$ is such that $Q(\overline{X}_{\mathbf{e}_q})\geq 0$, i.e. we should choose $a^*=a(\nu)$.
\end{proof}
\begin{rem}
In \cite{Novikovshiryaev2} the authors also consider the payoff function
 $G(x)=1-e^{-x^+}$ (see also Exercise 11.2 in \cite{ky}).
 For this payoff function we also readily deduce that for $y>x>0$ (stopping is not optimal when $X_t<0$)
 \begin{eqnarray*}
 V(y)-V(x)&\leq & \mathbb{E}[e^{-q\tau^*(y)}( G(y+X_{\tau^*(y)})-G(x+X_{\tau^*(y)}))]\\
 &=& \mathbb{E}[e^{-q\tau^*(y)}(e^{-(x+X_{\tau^*(y)})^+}-e^{-(y+X_{\tau^*(y)})^+})]\\
 &\leq &e^{-x}-e^{-y},
  \end{eqnarray*}
 from which it follows that the stopping region is again of the form $[x^*,\infty)$. The equivalent of Lemma 1 now is
 \[\mathbb{E}_x\left[e^{-q\tau_a^+}\left(1-e^{-X_{\tau_a^+}}\right)1_{\{\tau_a^+<\infty\}}\right]=\mathbb{E}_x\left[\left(1-\frac{e^{-\overline{X}_{\mathbf{e}_q}}}{\mathbb{E}[e^{-\overline{X}_{\mathbf{e}_q}}]}\right)1_{\{\overline{X}_{\mathbf{e}_q}\geq a\}}\right]\]
 from which we deduce immediately that $x^*=-\log(\mathbb{E}[e^{-\overline{X}_{\mathbf{e}_q}}])$.
\end{rem}
\section{The Shepp--Shiryaev optimal stopping problem}
The Shepp--Shiryaev optimal stopping problem (or Russian option) is given by
\begin{equation}V(x)=\sup_{\tau\in\mathcal{T}}\mathbb{E}[e^{-q\tau+(\overline{X}_\tau\vee x)}].\label{ss}\end{equation}
Here we assume that
$\psi(1):=\log(\mathbb{E}[e^{X_1}])<\infty$  and $q>(\psi(1)\vee 0)$ so that $V$ will be finite.
It was proposed and solved first for a Brownian motion in \cite{Shepp1}. Later, in \cite{Shepp2} an alternative method was described which was based on a change of measure
\[\left.\frac{d\mathbb{P}^1}{d\mathbb{P}}\right|_{\mathcal{F}_t}=e^{X_t-\psi(1)t}\]
under which (\ref{ss}) is transformed into an optimal stopping problem for the one-dimensional strong Markov process $Y_t^x=(x\vee \overline{X}_t)-X_t$ as
\begin{equation}
V(x)=\sup_{\tau\in\mathcal{T}}\mathbb{E}^1[e^{-(q-\psi(1))\tau+Y_\tau^x}],\label{ss2}\end{equation}
where $\mathbb{E}^1$ denotes the expectation under $\mathbb{P}^1$.
Similar to the previous cases we can show that the stopping set (for $Y$) is of the form $[z,\infty)$ for some $z\geq 0$. Indeed, for $y>x>0$ we get
\begin{eqnarray*}
V(y)-V(x)&\leq &\mathbb{E}\left[e^{-q\tau^*(y)}\left(e^{y\vee\overline{X}_{\tau^*(y)}}-e^{x\vee\overline{X}_{\tau^*(y)}}\right)\right]\\
&\leq &e^y-e^x.
\end{eqnarray*}
To get an expression for the expected payoff corresponding to stopping times $\tau_x^+$  the method we used earlier does not work now since the reflected process does not have independent increments. Instead we shall consider spectrally negative L\'evy processes (i.e. those with no positive jumps and the paths of which are not monotone). When $X$ is of bounded variation we shall denote the drift by $\mathrm{d}$. When $q\geq \mathrm{d}$, (\ref{ss}) is trivial since in this case $-qt+(\overline{X}_t\vee x)$ is a decreasing process and hence stopping immediately is optimal. Therefore we shall assume
\begin{equation}q<\mathrm{d}\quad\mbox{ when $X$ is of bounded variation. }\label{drift}\end{equation}
Scale functions are ubiquitous when it comes to fluctuation theory for spectrally negative L\'evy processes. For $r\geq 0$ define the scale function $W^{(r)}(x)$ on $[0,\infty)$ as the unique continuous function such that
\[\int_0^\infty e^{-\lambda x}W^{(r)}(x)\,dx=\frac{1}{\psi(\lambda) -r}\]
for $\lambda\geq 0$ such that $\psi(\lambda)>r$. We set $W^{(r)}(x)=0$ for $x<0$. Furthermore, define
$Z^{(r)}(x)=1+r\int_0^x W^{(r)}(y)\,dy.$ The following result was proved in \cite{avram2}, again using a verification lemma.
\begin{theorem}
Let $X$ be a spectrally negative L\'evy process satisfying (\ref{drift}) and let $q>\psi(1)\vee 0$.
The stopping set for (\ref{ss}) is given by $[x^*,\infty)$ where $x^*$ is the unique solution to the equation $Z^{(q)}(x)=qW^{(q)}(x)$ and
\[V(x)=e^{x}Z^{(q)}(x^*-x).\]
\end{theorem}
\begin{proof}
Having already established that the stopping set is of the form $[x,\infty)$, it suffices now to maximise the the expected payoff corresponding to first passage times $T^x_z=\inf\{t>0 : Y_t^x>z\}.$ In \cite{avram2} it was shown that
\[V(x,z):=\mathbb{E}^{(1)}[e^{-q T_z^x+Y^x_{T_z^x}}]=e^x\left(Z^{(q)}(z-x)-W^{(q)}(z-x)\frac{q W^{(q)}(z)-Z^{(q)}(z)}{W^{(q)^\prime}(z)-W^{(q)}(z)}\right).\]
Here we have implicitly assumed that the L\'evy measure of $X$ has no atoms when $X$ is of bounded variation, since otherwise $W^{(q)}$ will be differentiable only almost everywhere and we would have to resort to left derivatives instead.

To maximise over $z$ we are free to choose $x=0$ We find
\[f(z):=V(0,z)=\frac{Z^{(q)}(x)W^{(q)^\prime}(z)-q(W^{(q)}(z))^2}{W^{(q)^\prime}(z)-W^{(q)}(z)}.\]
For notational convenience we assume that $W^{(q)}$ is twice differentiable on $(0,\infty)$ (note that from (8.22) and (8.23) in \cite{ky} it follows that $W^{(q)}(x)$ is log-concave and hence twice differentiable almost everywhere. In fact, $W^{(q)}$ is a $C^2$ function on $(0,\infty)$ when $X$ has a Gaussian component, see \cite{kuz}). We then find that
\[f'(z)=\frac{(Z^{(q)}(z)-qW^{(q)}(z))(W^{(q)^\prime}(x)-W^{(q)}(z)W^{(q)^{\prime\prime}}(z))}{(W^{(q)^\prime}(z)-W^{(q)}(z))^2}.\]
Under the assumption (\ref{drift}) and $q>\psi(1)$ it holds that $g(z):=Z^{(q)}(z)-qW^{(q)}(z)$ satisfies $g(0)>0, g'(z)<0$ and $g(\infty)=-\infty$ so there is a unique $x^*$ such that $g(x^*)=0$. The log-concavity of $W^{(q)}$ allows us to deduce that $f'(z)\geq 0$ for $x<x^*$ and $f'(z)\leq 0$ for $x>x^*$.  It is therefore optimal to choose $z=x^*$ leading to $V(x)=e^x Z^{(q)}(x^*-x)$.
\end{proof}


\begin{thebibliography}{99}

\bibitem{aky} \sc L. Alili and A.E. Kyprianou: \rm Some remarks on the first passage of L\'evy processes, the
    American put and smooth pasting. {\it Ann. Appl.  Probab.} {\bf 15}, 2062-2080. (2005).
\bibitem{avram} \sc F. Avram, T. Chan and M. Usabel \rm On the valuation of costant barrier options under spectrally one-sided exponential L\'evy models and Carr's approximation for American puts. {\it Stochast. Process. Appl.} {\bf 100} 75--107. (2002).

\bibitem{avram2} \sc F. Avram, A.E. Kyprianou and M.R. Pistorius \rm Exit problems for spectrally negative L\'evy processes and applications to (Canadized) Russian options. {\it Ann. Appl. Probab.} {\bf 14} 215--238. (2004).

\bibitem{Baurdoux}\sc E.J. Baurdoux \rm Examples of optimal stopping via measure transformation for processes with one-sided jumps. {\it Stochastics} {\bf79} 303--307. (2007)

\bibitem{BL} \sc M. Beibel and H.R. Lerche: \rm A new look at optimal stopping problems related to mathematical frinance. {\it Statist. Sinica} {\bf 7} 93--108. (1997)

\bibitem{Chan}
\sc T. Chan \rm Some applications of {L}{\'{e}}vy processes in insurance and finance",
{\it Finance. Revue de l'Association Fran{\c{c}}aise de Finance}
{\bf 25}
    71--94. (2004)

\bibitem{Lerche-Urusov}\sc H.R. Lerche and M. Urusov: \rm Optimal stopping via measure transformation: the Beibel–-Lerche approach {\it Stoch. Stoch. Rep.} {\bf 79} 275--291. (2007)

\bibitem{Be} \sc J.~Bertoin: {\it L\'evy Processes.} \rm
Cambridge University Press, Cambridge, (1996).

\bibitem{bich} \sc K. Bichteler : {\it Stochastic Integration with Jumps.} \rm
Cambridge University Press, Cambridge, (2002).

\bibitem{Christensen}\sc S.Christensen,  A. Irle and A.A. Novikov: \rm  An Elementary Approach to Optimal Stopping Problems for AR(1) Sequences, {\it Sequential Analysis} {\bf 30}, 79 - 93 (2011).

\bibitem{Darling} \sc D.A. Darling, T. Liggett and H.M. Taylor \rm Optimal stopping for partial sums. {\it Ann. Math. Stat.} {\bf 43} 1363--138. (1972).

\bibitem{Gapeev}\sc P.V. Gapeev. and C. K\"uhn, C. \rm Perpetual convertible bonds in jump-diffusion models. {\it Statist. Decisions} {\bf 23} 15--31. (2005).

    \bibitem{kuz} \sc A. Kuznetsov, A.E. Kyprianou and V. Rivero
     {\it The theory of scale functions for spectrally negative L\'evy processes.}
L\'evy Matters II, Springer Lecture Notes in Mathematics.

\bibitem{ky} \sc A.E. Kyprianou: \it Introductory lectures of
fluctuations of L\'evy processes with applications. \rm Springer, Berlin, (2006).

\bibitem{Surya}
\sc A.E. Kyprianou and B.A. Surya: \rm On the {N}ovikov--{S}hiryaev optimal stopping problems in continuous time
{\it Electron. Comm. Probab.}{\bf 10} 146--154. (2005).


\bibitem{McKean}
\sc H. McKean \rm A free boundary problem for the heat equation arising from a problem of mathematical economics   {\it Ind. Manag. Rev.}    {\bf 6}
    32--39 (1965)



\bibitem{Mordecki}
\sc E. Mordecki: \rm Optimal stopping and perpetual options for {L}{\'{e}}vy processes
{\it Finance Stoch.}{\bf 6} 473--493. (2002).


\bibitem{Novikovshiryaev}\sc A.A. Novikov and A.N. Shiryaev. \rm On an effective solution to the optimal stopping problem for random walks. {\it Theory Probab. Appl.} {\bf 48} 288--303. (2004).

\bibitem{Novikovshiryaev2}\sc A.A. Novikov and A.N. Shiryaev. \rm  On a solution of the optimal stopping problem for processes with independent increments.
{\it Stochastics} {\bf 79} 393--406 (2007)

 \bibitem{PS} \sc G. Peskir and A. Shiryaev:  {\it Optimal Stopping and Free Boundary Value Problems.} \rm Birkh\"auser Verlag, Basel, (2006).

\bibitem{Salminen} \sc P. Salminen: \rm Optimal stopping, Appell polynomials, and WienerÐHopf factorization {\it Stochastics} {\bf 83} 611--622. (2011).

\bibitem{Shepp1}\sc L.A. Shepp and A.N. Shiryaev: \rm The Russian option: reduced regret/ {\it Ann. Appl. Probab.} {\bf 3} 603--631. (1993).

\bibitem{Shepp2} \sc L.A. Shepp and A.N. Shiryaev: \rm A new look at pricing the ``Russian option". {\it Theory Probab. Appl.} {\bf 39} 103--119. (1994).


\end{thebibliography}
\end{document}